\documentclass{amsart}

\usepackage{geometry}

\usepackage{amsmath,amssymb,amsthm}
\usepackage{xcolor}
\usepackage{graphicx} 

\makeatletter
\@namedef{subjclassname@2020}{%
  \textup{2020} Mathematics Subject Classification}
\makeatother

\newtheorem{thm}{Theorem}[section]

\newtheorem{prop}[thm]{Proposition}
\newtheorem{corollary}[thm]{Corollary}
\newtheorem{defn}[thm]{Definition}
\newtheorem{remark}[thm]{Remark}
\newtheorem{example}[thm]{Example}

\numberwithin{equation}{section}

\newcommand{\R}{\mathbb{R}}

\newcommand{\cG}{\mathcal{G}}

\newcommand{\td}{\tilde{d}}
\newcommand{\tW}{\widetilde{W}}
\renewcommand{\vec}[1]{\mathbf{#1}}

\everymath{\displaystyle}

\author{Fatma Terzioglu and Ryan Murray}
\address{Department of Mathematics, North Carolina State University, Raleigh, NC 27695}
\email[Corresponding author]{fterzioglu@ncsu.edu}
\email{rwmurray@ncsu.edu}
\date{}

\title[Stability of the $k$-Plane Transform on Measures]{Stability of the $k$-Plane Transform on Measures and H\"older-Type Comparisons of Wasserstein Metrics}

\subjclass[2020]{44A12, 46F12, 65R10}

\keywords{k-plane transform, Radon measures, stability estimates, Fourier distance, Wasserstein distance, max-sliced Wasserstein distance}

\begin{document}

\begin{abstract}
We establish stability estimates for the $k$-plane transform on finite positive Radon measures, with emphasis on Fourier and Wasserstein metrics. We first introduce a metric on $k$-plane transform data and prove a bi-Lipschitz stability estimate showing that this metric is equivalent to a generalized Fourier metric obtained by augmenting the Fourier distance between centered normalized measures with separate barycenter and total mass difference terms.

Building on a H\"older-type comparison between Fourier and Wasserstein metrics due to Carrillo and Toscani, we extend this comparison to positive Radon measures under uniform bounds on centered moments of order slightly larger than $2$. This yields H\"older-type stability for the $k$-plane transform in a generalized $2$-Wasserstein metric and, in particular, a $W_2$-stability estimate for centered probability measures.

We also compare the $2$-Wasserstein distance with its max-sliced analogue. For centered probability measures with uniformly bounded moments of order slightly larger than $2$, we prove a two-sided H\"older-type comparison between these distances. We then extend the result to positive Radon measures by applying it to centered normalized measures and adding separate barycenter and mass terms.

Finally, for absolutely continuous compactly supported probability measures with bounded densities, we prove a strong equivalence between the $2$-Wasserstein distance of the measures and the $(k/2-1)$-order Sobolev norm of the $k$-plane transform data of the difference of their densities.
\end{abstract}

\maketitle
\tableofcontents

\section{Introduction}\label{s:section1}
The $k$-plane transform plays a central role in integral geometry and tomography, and has more recently also proved useful in optimal transport, especially in connection with projection-based Wasserstein distances. Given a function $f$ on $\mathbb R^d$, its $k$-plane transform is defined by
\begin{align}\label{Pdef}
Pf(\alpha,y)=\int_\alpha f(y+x)\,dx,
\end{align}
where $\alpha \in G_{k,d}$, $y\in \alpha^\perp$, and $G_{k,d}$ denotes the Grassmannian of all $k$-dimensional linear subspaces of $\mathbb R^d$. The pair $(\alpha,y)$ represents the affine $k$-plane $y+\alpha$, and the collection of all such planes is the affine Grassmannian
$$
\mathcal G_{k,d}=\{(\alpha,y):\alpha\in G_{k,d},\ y\in \alpha^\perp\}.
$$

Two important special cases are the X-ray transform, corresponding to $k=1$ and integration over lines, and the classical Radon transform, corresponding to $k=d-1$ and integration over hyperplanes. Both arise in imaging science, especially in computerized tomography, where one seeks to recover an unknown function from its integrals over lines or planes. Because such inverse problems are ill-posed, stability estimates are essential for quantifying how measurement errors propagate into the reconstruction; see, for instance, \cite{Natterer}.

For compactly supported Sobolev functions, quantitative stability estimates for the $k$-plane transform have recently been established by the first author in \cite{Terzioglu2026}, extending classical results of Natterer \cite{Natterer} regarding X-ray and Radon transforms; see also \cite{Hahn,Hertle,Louis}. More precisely, if
$$
\|f\|_{H^s(\mathbb R^d)}
=
\left(
\int_{\mathbb R^d} |\widehat f(\xi)|^2(1+|\xi|^2)^s\,d\xi
\right)^{1/2}, \qquad s\in \mathbb R,
$$
and
\begin{align}\label{eq:sobolev-norm-G}
\|u\|_{H^s(\mathcal G_{k,d})}^2
=
\int_{G_{k,d}}\int_{\alpha^\perp}
(1+|\eta|^2)^s
|\widehat u(\alpha,\eta)|^2
\,d\eta\,d\alpha,
\end{align}
where $\widehat u(\alpha,\eta)$ denotes the Fourier transform of $u(\alpha,y)$ in the fiber variable $y$, then the following estimate holds.

\begin{thm}[\cite{Terzioglu2026}, Theorem 4.1]\label{kplaneSobolevEstimate}
Let $s\in\mathbb R$ and suppose that $f\in H^s(\mathbb R^d)$ satisfies $\operatorname{supp}(f)\subseteq \overline{\Omega}$ for some bounded open set $\Omega\subset\mathbb R^d$. Then there exist constants $c_{s,d,k},C_{s,d,k}>0$ such that
\begin{align}\label{StabilitySobolev}
c_{s,d,k}\|f\|_{H^s(\mathbb R^d)}
\le
\|Pf\|_{H^{s+k/2}(\mathcal G_{k,d})}
\le
C_{s,d,k}\|f\|_{H^s(\mathbb R^d)}.
\end{align}
\end{thm}

This estimate shows, in particular, that the $k$-plane transform gains $k/2$ derivatives on the Sobolev scale. Extensions to more general function spaces have also been considered; for the Radon transform see \cite{Kindermann2025,Sharafutdinov2017tensor,Sharafutdinov2021radon,GGG}, for the X-ray transform see \cite{Sharafutdinov2021x-ray}, and for the $k$-plane transform see \cite{Terzioglu2026}. 

Recently, Parhi and Unser \cite{Parhi2024} studied the $k$-plane transform and its dual in the space of distributions, proving invertibility results for the backprojection operator on suitable Banach spaces and applying these results to the regularization of inverse problems. 

In the setting of measures, Hahn and Quinto \cite{Hahn} studied the Radon transform of probability measures and finite signed measures, and showed that distances between such measures can be estimated from below and above by distances between their Radon transforms. Their analysis was carried out in the Prohorov metric, the dual bounded-Lipschitz metric, and other metrics that metrize weak convergence of probability measures.

In a related direction, the present paper considers stability questions for the $k$-plane transform of measures, with particular emphasis on quantitative comparison with Fourier and Wasserstein metrics. This is motivated both by analysis and by applications. In many problems, the underlying objects are more appropriately modeled by measures than by Sobolev functions; examples include sparse images, point sources, and particle distributions. At the same time, Wasserstein distances have become increasingly useful in imaging, inverse problems, and data science \cite{PeyreCuturi2019,BurgerSchonlieb2020}, because they provide a natural geometric framework for measures \cite{Villani2003,PanaterosZemel2019} and often exhibit more favorable variational properties, including improved convexity in optimization-based approaches to inverse problems \cite{EngquistRenYang2020}.

A related extension of Fourier-based metrics to probability measures with different barycenters was considered in \cite{Toscani2020}, where the authors introduce a translated version of the Fourier metric $d_2$ and study its use in imaging problems. Their main focus is the discrete setting, where finite-dimensional norm equivalences yield explicit comparisons with Wasserstein distances. In contrast, our results concern positive Radon measures with arbitrary total mass and their $k$-plane data. We use the continuous-space comparison between $d_2$ and $W_2$ in its H\"older form, following Carrillo--Toscani \cite{Carrillo2007}, and extend it by adding separate terms for the centered normalized component, barycenter difference, and mass difference.

Throughout the paper, we write $A\simeq B$ when there exist constants $c,C>0$ such that $cA\le B\le CA$. Our main contributions are as follows.
\begin{enumerate}
\item Our first main result is a bi-Lipschitz stability estimate for the $k$-plane transform on positive Radon measures. More precisely, we introduce a metric $D(P\mu,P\nu)$ on $k$-plane data (Definition \ref{d:k-planeDist}) and prove in Theorem \ref{t:kplaneEq_d2_general} that, for positive Radon measures $\mu,\nu$ with finite second moment,
$$
D(P\mu,P\nu)\simeq \widetilde d_2(\mu,\nu),
$$
where $\widetilde d_2$ is a generalized Fourier metric obtained from the classical Fourier metric \cite{Carrillo2007} by combining the $d_2$-distance between the centered normalized measures with the differences in barycenter and total mass; see Definition \ref{d:d2tilde}. For centered probability measures $\mu,\nu$, one has the exact identity
$$
D(P\mu,P\nu)=d_2(\mu,\nu),
$$
by Proposition \ref{p:PdistEqWass}.

\item Our second main result concerns the comparison between Fourier and Wasserstein metrics. For probability measures with finite moments up to order $2+\rho$, $\rho >0$, Carrillo and Toscani \cite{Carrillo2007} established a H\"older-type comparison between $d_2$ and $W_2$:
$$
c\, d_2(\mu,\nu) \le W_2(\mu,\nu) \le C\, d_2(\mu,\nu)^q, \qquad 0<q \le 1/2.
$$
In the present work, we prove an analogue for positive Radon measures by combining the corresponding estimate for the centered normalized measures with separate terms accounting for differences in barycenter and total mass; see Proposition \ref{p:td2tW2Holder}. Combined with the bi-Lipschitz stability estimate for $D(P\mu,P\nu)$, this yields a H\"older-type stability estimate for the $k$-plane transform in terms of the generalized Wasserstein distance:
$$
c\,\min\bigl\{\widetilde W_2(\mu,\nu)^{1/q},\widetilde W_2(\mu,\nu)\bigr\}
\le
D(P\mu,P\nu)
\le
C\,\widetilde W_2(\mu,\nu).
$$

\item Projection-based Wasserstein distances have received considerable attention in optimal transport and machine learning because they are often more tractable computationally than Wasserstein distances on the ambient space (see e.g., \cite{Deshpande2019, PatyCuturi2019}). For $\mu,\nu\in\mathcal P_2(\mathbb R^d)$, the max-sliced (or projection robust) 2-Wasserstein distance is defined by
$$
MSW_2(\mu,\nu)
:=
\sup_{\alpha\in G_{k,d}} W_2(P_\alpha\mu,P_\alpha\nu).
$$

In the $W_1$ case, Bayraktar and Guo \cite{BayraktarGuo2021} proved that $MSW_1\simeq W_1$, while Carlier, Figalli, M\'erigot, and Wang \cite{Carlier2025} obtained sharp one-sided comparisons for sliced $W_1$, including $k$-plane slicing. In the $W_2$ case, Paty and Cuturi \cite{PatyCuturi2019} proved a strong equivalence between $W_2$ and the subspace-robust distance $S_k$, a min-max relaxation of a projection-based transport problem.

As our third main result, we consider centered probability measures with uniformly bounded $(2+\rho)$-moments, and prove a two-sided H\"older-type comparison of the form
$$
c\,W_2(\mu,\nu)^{1/q}\le MSW_2(\mu,\nu)\le C\,W_2(\mu,\nu)^q,
\qquad 0<q \le 1/2.
$$

We also extend this comparison to positive Radon measures by combining the corresponding estimate for centered normalized measures with separate terms accounting for differences in barycenter and total mass, obtaining a corresponding two-sided H\"older-type comparison between $MS\widetilde W_2$ and $\widetilde W_2$; see Theorem \ref{maxSW-W}.

\item Finally, for absolutely continuous compactly supported probability measures with densities bounded above and below, in Theorem \ref{thm:kplaneW2densities}, we obtain the bi-Lipschitz estimate
$$
\|Pf-Pg\|_{H^{k/2-1}(\mathcal G_{k,d})}\simeq W_2(f\,dx,g\,dx).
$$
\end{enumerate}

The paper is organized as follows. In Section 2, we develop the $k$-plane transform for positive Radon measures and establish its basic properties. In Section 3, we introduce the generalized Fourier metric $\widetilde d_2$ and the metric $D(P\mu,P\nu)$ on $k$-plane data, and prove the corresponding bi-Lipschitz stability estimate. In Section 4, we compare Fourier and Wasserstein metrics and derive H\"older-type estimates for positive Radon measures through generalized metrics that include separate barycenter and mass terms. In Section 5, we use Fourier-based distance estimates for the $k$-plane transform data to establish H\"older-type comparisons between max-sliced and Wasserstein distances for probability measures, together with corresponding generalized analogs for positive Radon measures. In the final section, we specialize to compactly supported probability densities and obtain a bi-Lipschitz estimate for $k$-plane data in the $2$-Wasserstein metric.

\section{The $k$-plane transform of measures and its basic properties}
In this section we review the definition and basic properties of the $k$-plane transform. As we are interested in measure-valued data, we develop our definitions using distribution theory, following the same approach as in \cite{gel2014integral} for the Radon Transform.

Let $C_0(\mathbb R^d)$ denote the space of continuous functions on $\mathbb R^d$ that vanish at infinity, equipped with the supremum norm $\|\cdot\|_\infty$. We write $\mathcal M(\mathbb R^d)$ for the dual of $C_0(\mathbb R^d)$, that is, the space of all continuous linear functionals on $C_0(\mathbb R^d)$. By the Riesz representation theorem (see e.g., Theorem 7.2, \cite{Folland}), $\mathcal M(\mathbb R^d)$ may be identified with the space of finite Radon measures on $\mathbb R^d$. Accordingly, for $\mu \in \mathcal M(\mathbb R^d)$ and $\varphi \in C_0(\mathbb R^d)$, we write
$$
\langle \mu,\varphi\rangle = \int_{\mathbb R^d} \varphi\, d\mu.
$$

We denote by $\mathcal M^+(\mathbb R^d)$ the convex cone of finite positive Radon measures in $\mathcal M(\mathbb R^d)$ with positive total mass, that is,
$$
0 < M_\mu := \mu(\mathbb R^d) < \infty.
$$

For $m \ge 0$, we denote by $\mathcal M_m^+(\mathbb R^d)$ the set of measures $\mu \in \mathcal M^+(\mathbb R^d)$ with finite absolute moments up to order $m$, namely,
$$
\mathcal M_m^+(\mathbb R^d)
:=
\left\{
\mu \in \mathcal M^+(\mathbb R^d)
:
\int_{\mathbb R^d} |x|^m\, d\mu(x) < \infty,
\right\}.
$$

We note that since $M_\mu<\infty$, finiteness of the $m$-th absolute moment also implies finiteness of all lower-order absolute moments. Indeed, for every $0\le m'\le m$,
$$
|x|^{m'}\le 1+|x|^m,
\qquad x\in \mathbb R^d.
$$
Therefore, if $\mu\in \mathcal M_m^+(\mathbb R^d)$, then
$$
\int_{\mathbb R^d}|x|^{m'}\,d\mu(x)
\le
\int_{\mathbb R^d}(1+|x|^m)\,d\mu(x)
=
M_\mu+\int_{\mathbb R^d}|x|^m\,d\mu(x)
<\infty
$$
for all $0\le m'\le m$.

For $\mu \in \mathcal M_1^+(\mathbb R^d)$, the barycenter of $\mu$ is defined by
$$
\vec{m}_\mu
:=
\frac{1}{M_\mu}\int_{\mathbb R^d} x\, d\mu(x).
$$

We denote by $\mathcal P_m(\mathbb R^d)$ the space of probability measures with finite absolute moments up to order $m$, namely,
$$
\mathcal P_m(\mathbb R^d)
:=
\{\mu \in \mathcal M_m^+(\mathbb R^d) : \mu(\mathbb R^d)=1\}.
$$
If $\mu \in \mathcal M_m^+(\mathbb R^d)$, then its normalized measure
$$
\bar\mu := \frac{1}{M_\mu}\mu,
$$
belongs to $\mathcal P_m(\mathbb R^d)$, and satisfies
$$
\langle \bar\mu,\varphi\rangle
=
\frac{1}{M_\mu}\langle \mu,\varphi\rangle,
\qquad \varphi \in C_0(\mathbb R^d).
$$

For $a\in \mathbb R^d$, we define the shift of $\mu$ by $a$ by
\begin{align}\label{shift}
\langle T_a\mu,\varphi\rangle
=
\langle \mu, T_{-a}\varphi\rangle,
\qquad
\varphi \in C_0(\mathbb R^d),
\qquad
T_{-a}\varphi(x)=\varphi(x-a).
\end{align}

We also denote by
$$
\mu_0 := T_{\vec{m}_\mu}\mu
$$
the centered version of $\mu$. In particular, $\mu_0$ has barycenter at the origin.

Let $G_{k,d}$ denote the Grassmannian of all $k$-dimensional linear subspaces of $\mathbb R^d$. We define the space of affine $k$-planes in $\mathbb R^d$ by
$$
\mathcal G_{k,d}
=
\{(\alpha,y): \alpha \in G_{k,d},\ y\in \alpha^\perp\},
$$
where the pair $(\alpha,y)$ represents the affine $k$-plane $y+\alpha$. For $\alpha \in G_{k,d}$, we let $\pi_{\alpha^\perp}:\mathbb R^d \to \alpha^\perp$ denote the orthogonal projection onto $\alpha^\perp$.

\begin{defn}
For $\psi \in C_0(\mathcal G_{k,d})$, we define the backprojection operator by
\begin{align}\label{kplaneAdj}
P^*\psi(x)
=
\int_{G_{k,d}} \psi(\alpha,\pi_{\alpha^\perp}(x))\, d\alpha,
\end{align}
where $d\alpha$ is the canonical $O(d)$-invariant probability measure on $G_{k,d}$, obtained as the pushforward of the Haar measure on $O(d)$ under the quotient map $O(d)\to G_{k,d}$.
\end{defn}

If $\psi \in C_0(\mathcal G_{k,d})$, then $P^*\psi \in C_0(\mathbb R^d)$; see the Appendix for a proof. We can now define the $k$-plane transform as a type of adjoint of the backprojection operator.

\begin{defn}
The $k$-plane transform (or projection) of $\mu \in \mathcal M(\mathbb R^d)$ is the measure $P\mu \in \mathcal M(\mathcal G_{k,d})$ defined by
\begin{align}\label{kplane}
\langle P\mu,\psi\rangle
=
\langle \mu, P^*\psi\rangle,
\qquad
\psi \in C_0(\mathcal G_{k,d}),
\end{align}
where $P^*$ is given by \eqref{kplaneAdj}.
\end{defn}

For a fixed $\alpha \in G_{k,d}$, we also consider the projection of $\mu$ onto $\alpha^\perp$.

\begin{defn}
Let $\mu \in \mathcal M(\mathbb R^d)$ and $\alpha \in G_{k,d}$. We define
$$
P_\alpha\mu := \pi_{\alpha^\perp}\#\mu,
$$
the pushforward of $\mu$ under $\pi_{\alpha^\perp}$. Equivalently,
\begin{align}\label{kplaneT}
\langle P_\alpha\mu,\phi\rangle
=
\langle \mu,\phi\circ \pi_{\alpha^\perp}\rangle,
\qquad
\phi \in C_0(\alpha^\perp).
\end{align}
\end{defn}

Since $\pi_{\alpha^\perp}:\mathbb R^d\to \alpha^\perp$ is continuous, the pushforward $P_\alpha\mu=\pi_{\alpha^\perp}\#\mu$ is a finite Radon measure on $\alpha^\perp$ for every $\mu\in \mathcal M(\mathbb R^d)$ and every $\alpha\in G_{k,d}$.

The following proposition collects several basic properties of the $k$-plane transform that will be used throughout the paper.

\begin{prop}[Basic properties of the $k$-plane transform]\label{p:kplaneProperties}
Let $\alpha \in G_{k,d}$ and $\mu \in \mathcal M_1^+(\mathbb R^d)$. Then:
\begin{enumerate}
\item[(i)] $P_\alpha\mu$ preserves total mass:
$$
M_{P_\alpha\mu}=M_\mu.
$$

\item[(ii)] If $\mu \in \mathcal M_m^+(\mathbb R^d)$, $m \ge 0$, then
$$
P_\alpha\mu \in \mathcal M_m^+(\alpha^\perp), \quad \alpha \in G_{k,d}.
$$

\item[(iii)] For any $\mu \in \mathcal M_1^+(\mathbb R^d)$ the barycenter of the projection is the projection of the barycenter:
$$
\vec{m}_{P_\alpha\mu} = \pi_{\alpha^\perp}(\vec{m}_\mu).
$$

\item[(iv)] $P_\alpha$ commutes with shifts:
$$
P_\alpha(T_a\mu)=T_{\pi_{\alpha^\perp}(a)}(P_\alpha\mu),
\qquad a\in \mathbb R^d.
$$

\item[(v)] Normalization and recentering commute with projection:
$$
(\overline{P_\alpha\mu})_0 = P_\alpha\bar\mu_0.
$$
\end{enumerate}
\end{prop}

\begin{proof}

\begin{enumerate}
\item[(i)] By definition of pushforward,
$$
M_{P_\alpha\mu}
=
P_\alpha\mu(\alpha^\perp)
=
\mu\big(\pi_{\alpha^\perp}^{-1}(\alpha^\perp)\big)
=
\mu(\mathbb R^d)
=
M_\mu.
$$

\item[(ii)] If $\mu \in \mathcal M_m^+(\mathbb R^d)$, then
$$
\int_{\alpha^\perp} |y|^m\, d(P_\alpha\mu)(y)
=
\int_{\mathbb R^d} |\pi_{\alpha^\perp}(x)|^m\, d\mu(x)
\le
\int_{\mathbb R^d} |x|^m\, d\mu(x)
<
\infty.
$$
Hence $P_\alpha\mu \in \mathcal M_m^+(\alpha^\perp)$.

\item[(iii)] Using the definition of the pushforward, we obtain
$$
\int_{\alpha^\perp} y\, d(P_\alpha\mu)(y)
=
\int_{\mathbb R^d} \pi_{\alpha^\perp}(x)\, d\mu(x).
$$
We note that these integrals are finite by (ii).
Therefore, by part (i),
$$
\vec{m}_{P_\alpha\mu}
=
\frac{1}{M_{P_\alpha\mu}}
\int_{\alpha^\perp} y\, d(P_\alpha\mu)(y)
=
\frac{1}{M_\mu}
\int_{\mathbb R^d} \pi_{\alpha^\perp}(x)\, d\mu(x).
$$
Since $\pi_{\alpha^\perp}$ is a continuous linear map and $\mu$ has finite first moment, we may pass $\pi_{\alpha^\perp}$ through the vector-valued integral:
$$
\frac{1}{M_\mu}
\int_{\mathbb R^d} \pi_{\alpha^\perp}(x)\, d\mu(x)
=
\pi_{\alpha^\perp}\!\left(\frac{1}{M_\mu}\int_{\mathbb R^d} x\, d\mu(x)\right)
=
\pi_{\alpha^\perp}(\vec{m}_\mu),
$$
which proves the claim.

\item[(iv)] Let $\phi \in C_0(\alpha^\perp)$. Since $\pi_{\alpha^\perp}$ is linear,
$$
T_{-a}(\phi\circ \pi_{\alpha^\perp})(x)
=
\phi(\pi_{\alpha^\perp}(x-a))
=
\phi(\pi_{\alpha^\perp}(x)-\pi_{\alpha^\perp}(a))
=
(T_{-\pi_{\alpha^\perp}(a)}\phi)(\pi_{\alpha^\perp}(x)).
$$
Therefore,
\begin{align*}
\langle P_\alpha(T_a\mu),\phi\rangle
&=
\langle T_a\mu,\phi\circ \pi_{\alpha^\perp}\rangle 
=
\langle \mu, T_{-a}(\phi\circ \pi_{\alpha^\perp})\rangle 
=
\langle \mu, (T_{-\pi_{\alpha^\perp}(a)}\phi)\circ \pi_{\alpha^\perp}\rangle \\
&=
\langle P_\alpha\mu, T_{-\pi_{\alpha^\perp}(a)}\phi\rangle 
=
\langle T_{\pi_{\alpha^\perp}(a)}(P_\alpha\mu),\phi\rangle.
\end{align*}
Hence
$$
P_\alpha(T_a\mu)=T_{\pi_{\alpha^\perp}(a)}(P_\alpha\mu).
$$

\item[(v)] By part (i), and the linearity of $P_\alpha$,
$$
\overline{P_\alpha\mu}
=
\frac{1}{M_{P_\alpha\mu}}P_\alpha\mu
=
\frac{1}{M_\mu}P_\alpha\mu
=
P_\alpha\bar\mu.
$$
By parts (iii) and (iv),
$$
(P_\alpha\nu)_0
=
T_{\vec{m}_{P_\alpha\nu}}(P_\alpha\nu)
=
T_{\pi_{\alpha^\perp}(\vec{m}_\nu)}(P_\alpha\nu)
=
P_\alpha(T_{\vec{m}_\nu}\nu)
=
P_\alpha\nu_0,
$$
for any $\nu \in \mathcal M_1^+(\mathbb R^d)$. Applying this with $\nu=\bar\mu$, we get
$$
(\overline{P_\alpha\mu})_0
=
(P_\alpha\bar\mu)_0
=
P_\alpha\bar\mu_0,
$$
\end{enumerate}
which completes the proof.
\end{proof}

We next state the Fourier--slice theorem for measures, which gives the exact relation between the Fourier transform and the $k$-plane transform.

\begin{defn}\label{FourierT-def}
The Fourier transform of $\mu \in \mathcal M(\mathbb R^d)$ is defined by
\begin{align}\label{FourierT}
\widehat{\mu}(\xi)
=
\langle \mu, e^{-ix\cdot \xi}\rangle
=
\int_{\mathbb R^d} e^{-ix\cdot \xi}\, d\mu(x),
\qquad \xi \in \mathbb R^d.
\end{align}
\end{defn}
For $\mu \in \mathcal M(\mathbb R^d)$, the function $\widehat{\mu}$ is bounded and uniformly continuous on $\mathbb R^d$, as follows from the finiteness of $|\mu|$ and the dominated convergence theorem. From the definition of the shift operator \eqref{shift}, we also obtain
\begin{align}\label{FourierShift}
\widehat{T_a\mu}(\xi)
=
\langle T_a\mu, e^{-ix\cdot \xi}\rangle
=
\langle \mu, e^{-i(x-a)\cdot \xi}\rangle
=
e^{ia\cdot \xi}\widehat{\mu}(\xi).
\end{align}
For a measure $u_\alpha \in \mathcal M(\alpha^\perp)$, we define its Fourier transform on the fiber $\alpha^\perp$ by
\begin{align}\label{FourierTonG}
\widehat{u_\alpha}(\xi)
=
\langle u_\alpha, e^{-iy\cdot \xi}\rangle,
\qquad \xi \in \alpha^\perp.
\end{align}
Here the last duality pairing is over functions of $y$, with $y \in \alpha^\perp$.

\begin{prop}[Fourier--slice theorem]\label{p:FST}
Let $\mu \in \mathcal M(\mathbb R^d)$ and $1 \le k \le d-1$. Then, for every $\alpha \in G_{k,d}$,
\begin{align}\label{FST}
\widehat{P_\alpha\mu}(\xi)=\widehat{\mu}(\xi),
\qquad \xi \in \alpha^\perp.
\end{align}
\end{prop}

\begin{proof}
If $\xi \in \alpha^\perp$, then
$$
\pi_{\alpha^\perp}(x)\cdot \xi = x\cdot \xi,
\qquad x\in \mathbb R^d.
$$
Hence,
\begin{align*}
\widehat{P_\alpha\mu}(\xi)
=
\langle P_\alpha\mu, e^{-iy\cdot \xi}\rangle 
=
\langle \mu, e^{-i\pi_{\alpha^\perp}(x)\cdot \xi}\rangle 
=
\langle \mu, e^{-ix\cdot \xi}\rangle 
=
\widehat{\mu}(\xi).
\end{align*}
\end{proof}

The Fourier--slice theorem shows that the Fourier transform of the projected measure $P_\alpha\mu$ on $\alpha^\perp$ coincides with the restriction of the Fourier transform of $\mu$ to $\alpha^\perp$. This identity will play a central role in establishing the two-sided estimates for the $k$-plane transform below.

\section{Bi-Lipschitz stability estimates for the $k$-plane transform in Fourier metrics}
The Fourier distance compares probability measures in terms of their Fourier transforms. It is particularly well suited to the $k$-plane transform, since the Fourier-slice theorem identifies the Fourier transform of each projected measure with the restriction of the Fourier transform in the ambient space. We first recall this distance for the probability measures, then extend it to positive measures with arbitrary mass and barycenter.

\begin{prop}[\cite{Carrillo2007}, Proposition 2.6]
Let $s>0$. For any $\mu, \nu \in \mathcal P_s(\mathbb R^d)$ possessing equal moments up to order $\lceil s-1\rceil$,
\begin{align}\label{dist_s}
    d_s(\mu,\nu) = \sup_{\xi \in \mathbb R^d \setminus \{0\}} \frac{|\widehat{\mu}(\xi) - \widehat{\nu}(\xi)|}{|\xi|^s},
\end{align}
defines a distance.
\end{prop}

\begin{example}\label{ex:Gaussians_d2}
    As a running example, we consider two Gaussians $\mu_1,\mu_2$, with zero mean and covariance matrices $\Sigma_1,\Sigma_2$. The Fourier transforms of these measures are given by $\hat \mu_i(\xi) = \exp( -\frac{1}{2} \xi^T \Sigma_i \xi)$. We seek to compute $d_2(\mu_1,\mu_2)$, and so by letting $\xi = t \xi_0$ for a unit vector $\xi_0$ we seek to find the supremum, over $t>0$, of
    \[
    \frac{|\exp( -\frac{1}{2} t^2\xi_0^T \Sigma_1 \xi_0) - \exp( -\frac{1}{2}t^2 \xi_0^T \Sigma_2 \xi_0)|}{t^2} = \Big| \frac12 \int_{\xi_0^T \Sigma_1 \xi_0}^{\xi_0^T \Sigma_2 \xi_0} e^{-rt^2/2} dr \Big|.
    \]
   The above quantity is decreasing in $t$, and thus its supremum is $\frac{1}{2}|(\xi_0^T \Sigma_1 \xi_0 - \xi_0^T \Sigma_2 \xi_0)|$. Hence we have
    \[
    d_2(\mu_1,\mu_2) = \frac{1}{2}\sup_{\xi_0 \in \mathbb{S}^{d-1}}  |\xi_0^T (\Sigma_1  - \Sigma_2) \xi_0|.
    \]
    This then immediately gives $d_2(\mu_1,\mu_2) = \frac{1}{2} \|\Sigma_1 - \Sigma_2\|_2$, where by $\|\cdot\|_2$ we mean the matrix $2$-norm. 
\end{example}

In what follows, we restrict our attention to the case $s=2$. Since the measures considered here need not have unit mass or the same barycenter, we augment the above distance by terms accounting for differences in mass and barycenter. Recall that if $\mu \in \mathcal M_1^+(\mathbb R^d)$, then $M_\mu$ and $\vec{m}_\mu$ denote its total mass and barycenter, respectively, and
$$
\bar\mu_0 := \frac{1}{M_\mu} T_{\vec{m}_\mu}\mu,
$$
is the associated probability measure with zero barycenter.

\begin{defn}\label{d:d2tilde}
For $\mu, \nu \in \mathcal M_2^+(\mathbb R^d)$, we define
\begin{align}\label{d2tilde}
    \widetilde d_2(\mu,\nu)
    := d_2(\bar\mu_0,\bar\nu_0) + |\vec{m}_\mu - \vec{m}_\nu|
    + |M_\mu - M_\nu|.
\end{align}
\end{defn}

Since $\bar\mu_0$ and $\bar\nu_0$ are probability measures with zero barycenter, they have equal moments up to order $1$. Hence $d_2(\bar\mu_0,\bar\nu_0)$ is well-defined. Moreover, if $\mu,\nu\in \mathcal P_2(\mathbb R^d)$ and $\vec{m}_\mu = \vec{m}_\nu = \vec{0}$, then
$$
\widetilde d_2(\mu,\nu)=d_2(\mu,\nu).
$$

Furthermore, $\widetilde d_2$ defines a distance on $\mathcal M_2^+(\mathbb R^d)$. Indeed, positivity and symmetry are immediate. If $\widetilde d_2(\mu,\nu)=0$, then $M_\mu=M_\nu$, $\vec{m}_\mu=\vec{m}_\nu$, and $\bar\mu_0=\bar\nu_0$. Since
$$
\mu=M_\mu T_{-\vec{m}_\mu}\bar\mu_0,
$$
we obtain $\mu=\nu$. The triangle inequality follows by applying the triangle inequality to each of the three distance terms in \eqref{d2tilde}.

We remark that the choice of uniform weights in Definition \ref{d:d2tilde} is made only for notational simplicity. More generally, one may define
$$
\widetilde d_{2,a,b,c}(\mu,\nu)
:=
a\,d_2(\bar\mu_0,\bar\nu_0)
+
b|\vec{m}_\mu-\vec{m}_\nu|
+
c|M_\mu-M_\nu|,
\qquad a,b,c>0.
$$
This is again a distance on $\mathcal M_2^+(\mathbb R^d)$. All results below remain valid for this weighted version, with constants depending also on the fixed weights $a,b,c$.

Alternatively, one may use the distance
$$
\widetilde d_{2,a,b,c}^{\,\ell^2}(\mu,\nu)
:=
\left(
a^2d_2(\bar\mu_0,\bar\nu_0)^2
+
b^2|\vec{m}_\mu-\vec{m}_\nu|^2
+
c^2|M_\mu-M_\nu|^2
\right)^{1/2}.
$$
The corresponding statements and proofs are analogous. In the sequel we use the unweighted $\ell^1$-type version in Definition \ref{d:d2tilde} to keep the notation and estimates simple.

For probability measures, a related translated Fourier-based distance was introduced in \cite{Toscani2020} to compare measures with different barycenters. In contrast, we work with positive Radon measures of arbitrary finite total mass and combine the centered Fourier distance with the barycenter and mass differences. This formulation is well suited to the $k$-plane transform, since normalization, recentering, barycenters, and total mass interact naturally with projection, as shown in Proposition \ref{p:kplaneProperties}.

\begin{defn}\label{d:k-planeDist}
For $\mu, \nu \in \mathcal M_2^+(\mathbb R^d)$, we define the distance between their $k$-plane transforms by
\begin{align}\label{k-planeDist}
    D(P\mu,P\nu) := \sup_{\alpha \in G_{k,d}} \widetilde d_2(P_\alpha\mu,P_\alpha\nu).
\end{align}
\end{defn}

\begin{example}
Continuing as in Example \ref{ex:Gaussians_d2}, let $\mu_1,\mu_2$ be centered
Gaussian measures with covariance matrices $\Sigma_1,\Sigma_2$. For each
$\alpha\in G_{k,d}$, the pushforward of $\mu_i$ under the orthogonal projection
$\pi_{\alpha^\perp}$ is again a centered Gaussian measure on $\alpha^\perp$.
If $\Pi_{\alpha^\perp}$ denotes the orthogonal projection matrix onto
$\alpha^\perp$, its covariance, viewed as an operator on the embedded subspace
$\alpha^\perp\subset \mathbb R^d$, is
$\Pi_{\alpha^\perp}\Sigma_i\Pi_{\alpha^\perp}.$
Therefore, by the same computation as in Example \ref{ex:Gaussians_d2},
\[
D(P\mu_1,P\mu_2)
=
\frac12
\sup_{\alpha\in G_{k,d}}
\left\|
\Pi_{\alpha^\perp}(\Sigma_1-\Sigma_2)\Pi_{\alpha^\perp}
\right\|_2 .
\]
Since $\Pi_{\alpha^\perp}$ is an orthogonal projection, this immediately gives
\[
D(P\mu_1,P\mu_2)\le d_2(\mu_1,\mu_2).
\]
On the other hand, because $\Sigma_1-\Sigma_2$ is symmetric, its spectral norm
is attained on an eigenvector corresponding to an eigenvalue of largest
absolute value. Choosing $\alpha$ so that $\alpha^\perp$ contains this
eigenvector yields equality. Hence, in this Gaussian setting,
\[
D(P\mu_1,P\mu_2)=d_2(\mu_1,\mu_2).
\]
\end{example}

We now show, as a consequence of the Fourier--slice theorem, that the distance $D(P\mu,P\nu)$ is equivalent to the Fourier-based distance $\widetilde d_2(\mu,\nu)$. We first prove the result for probability measures with zero barycenter and then extend it to general measures in $\mathcal M_2^+(\mathbb R^d)$.

\begin{prop}\label{p:PdistEqWass}
Let $1 \le k \le d-1$ be an integer. For any $\mu,\nu \in \mathcal P_2(\mathbb R^d)$ with
$\vec{m}_\mu=\vec{m}_\nu=\vec{0}$,
\begin{align}\label{eq:proj-eq-d2}
    D(P\mu,P\nu)
    =
    \sup_{\alpha \in G_{k,d}} d_2(P_\alpha\mu,P_\alpha\nu)
    =
    d_2(\mu,\nu).
\end{align}
\end{prop}

\begin{proof}
Let $\mu,\nu \in \mathcal P_2(\mathbb R^d)$ with $\vec{m}_\mu=\vec{m}_\nu=\vec{0}$. By Proposition \ref{p:kplaneProperties}, for every $\alpha \in G_{k,d}$ we have
$$
P_\alpha\mu,\,P_\alpha\nu \in \mathcal P_2(\alpha^\perp)
\quad \text{and} \quad
\vec{m}_{P_\alpha\mu}=\vec{m}_{P_\alpha\nu}=\vec{0}.
$$
Hence, by Definition \ref{d:k-planeDist},
$$
D(P\mu,P\nu)=\sup_{\alpha \in G_{k,d}} d_2(P_\alpha\mu,P_\alpha\nu).
$$
Using the definition of $d_2$ together with the Fourier--slice theorem (Proposition \ref{p:FST}), we obtain
\begin{align*}
D(P\mu,P\nu)
&=
\sup_{\alpha \in G_{k,d}} \sup_{\xi \in \alpha^\perp \setminus \{0\}}
\frac{|\widehat{P_\alpha\mu}(\xi)-\widehat{P_\alpha\nu}(\xi)|}{|\xi|^2} \\
&=
\sup_{\alpha \in G_{k,d}} \sup_{\xi \in \alpha^\perp \setminus \{0\}}
\frac{|\widehat\mu(\xi)-\widehat\nu(\xi)|}{|\xi|^2}.
\end{align*}
Since
$$
\mathbb R^d \setminus \{0\}
=
\bigcup_{\alpha \in G_{k,d}} \left(\alpha^\perp \setminus \{0\}\right),
$$
it follows that
\begin{align*}
D(P\mu,P\nu)
&=
\sup_{\xi \in \mathbb R^d \setminus \{0\}}
\frac{|\widehat\mu(\xi)-\widehat\nu(\xi)|}{|\xi|^2}
=
d_2(\mu,\nu).
\end{align*}
This proves the claim.
\end{proof}

\begin{thm}\label{t:kplaneEq_d2_general}
Let $\mu,\nu \in \mathcal M_2^+(\mathbb R^d)$. Then
\begin{align}\label{eq:biLip-d2}
\frac{1}{2}\,\widetilde d_2(\mu,\nu)
\le
D(P\mu,P\nu)
\le
\widetilde d_2(\mu,\nu).
\end{align}
\end{thm}

\begin{proof}
We begin with the upper bound. By definition of $\td_2$ and Proposition \ref{p:kplaneProperties}, for every $\alpha \in G_{k,d}$,
\begin{align*}
\widetilde d_2(P_\alpha\mu,P_\alpha\nu)
&=
d_2\bigl((\overline{P_\alpha\mu})_0,(\overline{P_\alpha\nu})_0\bigr)
+
|\vec{m}_{P_\alpha\mu}-\vec{m}_{P_\alpha\nu}|
+
|M_{P_\alpha\mu}-M_{P_\alpha\nu}| \\
&=
d_2(P_\alpha\bar\mu_0,P_\alpha\bar\nu_0)
+
|\pi_{\alpha^\perp}(\vec{m}_\mu-\vec{m}_\nu)|
+
|M_\mu-M_\nu|.
\end{align*}
Taking the supremum over $\alpha \in G_{k,d}$, we obtain
\begin{align*}
D(P\mu,P\nu)
&=
\sup_{\alpha \in G_{k,d}} \widetilde d_2(P_\alpha\mu,P_\alpha\nu) \\
&\le
\sup_{\alpha \in G_{k,d}} d_2(P_\alpha\bar\mu_0,P_\alpha\bar\nu_0)
+
\sup_{\alpha \in G_{k,d}} |\pi_{\alpha^\perp}(\vec{m}_\mu-\vec{m}_\nu)|
+
|M_\mu-M_\nu|.
\end{align*}
Now
$$
\sup_{\alpha \in G_{k,d}} |\pi_{\alpha^\perp}(v)| = |v|, \qquad v\in \mathbb R^d,
$$
since $ |\pi_{\alpha^\perp}(v)| \le |v| $ for every $\alpha\in G_{k,d}$, and because $1 \le k \le d-1$, one can choose $\alpha \subset v^\perp$, in which case $v\in \alpha^\perp$ and therefore $\pi_{\alpha^\perp}(v) = v$.

Moreover, by Proposition \ref{p:PdistEqWass},
$$
\sup_{\alpha \in G_{k,d}} d_2(P_\alpha\bar\mu_0,P_\alpha\bar\nu_0)
=
d_2(\bar\mu_0,\bar\nu_0).
$$
Therefore,
\begin{align*}
D(P\mu,P\nu)
&\le
d_2(\bar\mu_0,\bar\nu_0)
+
|\vec{m}_\mu-\vec{m}_\nu|
+
|M_\mu-M_\nu|
=
\widetilde d_2(\mu,\nu).
\end{align*}

We next prove the lower bound in \eqref{eq:biLip-d2}. Since
$$
d_2(P_\alpha\bar\mu_0,P_\alpha\bar\nu_0) \ge 0
\qquad\text{and}\qquad
|\pi_{\alpha^\perp}(\vec{m}_\mu-\vec{m}_\nu)| \ge 0,
$$
for every $\alpha\in G_{k,d}$, we have
\begin{align*}
&\sup_{\alpha\in G_{k,d}} d_2(P_\alpha\bar\mu_0,P_\alpha\bar\nu_0)
+
\sup_{\alpha\in G_{k,d}} |\pi_{\alpha^\perp}(\vec{m}_\mu-\vec{m}_\nu)| \\
&\qquad\le
2\sup_{\alpha\in G_{k,d}}
\left(
d_2(P_\alpha\bar\mu_0,P_\alpha\bar\nu_0)
+
|\pi_{\alpha^\perp}(\vec{m}_\mu-\vec{m}_\nu)|
\right).
\end{align*}
Using this together with Proposition \ref{p:PdistEqWass} and definition of $\td_2$, we obtain
\begin{align*}
\widetilde d_2(\mu,\nu)
&=
d_2(\bar\mu_0,\bar\nu_0)
+
|\vec{m}_\mu-\vec{m}_\nu|
+
|M_\mu-M_\nu| \\
&=
\Bigg(
\sup_{\alpha\in G_{k,d}} d_2(P_\alpha\bar\mu_0,P_\alpha\bar\nu_0)
+
\sup_{\alpha\in G_{k,d}} |\pi_{\alpha^\perp}(\vec{m}_\mu-\vec{m}_\nu)|
\Bigg)
+
|M_\mu-M_\nu| \\
&\le
2\sup_{\alpha\in G_{k,d}}
\left(
d_2(P_\alpha\bar\mu_0,P_\alpha\bar\nu_0)
+
|\pi_{\alpha^\perp}(\vec{m}_\mu-\vec{m}_\nu)|
+
|M_\mu-M_\nu|\right)\\
& = 2\sup_{\alpha\in G_{k,d}} \widetilde d_2(P_\alpha\mu,P_\alpha\nu)\\
& = 2\,D(P\mu,P\nu),
\end{align*}
by the definition of the distances $\td_2$ and $D$. Hence,
$$
D(P\mu,P\nu)\ge \frac{1}{2}\,\widetilde d_2(\mu,\nu),
$$
which completes the proof.
\end{proof}

\section{H\"older-type stability estimates for the $k$-plane transform in $W_2$-type metrics}

In the previous section, we showed that the distance $D(P\mu,P\nu)$ on $k$-plane data is bi-Lipschitz equivalent to the generalized Fourier distance $\widetilde d_2(\mu,\nu)$. We now combine this fact with a H\"older-type comparison between the Fourier distance $d_2$ and the Wasserstein distance $W_2$ due to Carrillo and Toscani \cite{Carrillo2007}. This yields a corresponding H\"older type estimate for the $k$-plane transform in terms of a generalized Wasserstein distance.

\begin{defn}
For $\mu,\nu \in \mathcal P_2(\mathbb R^d)$, the $2$-Wasserstein distance is defined by
\begin{align}
W_2(\mu,\nu)
:=
\inf_{\pi \in \Pi(\mu,\nu)}
\left(
\int_{\mathbb R^d \times \mathbb R^d} |x-y|^2\, d\pi(x,y)
\right)^{1/2},
\end{align}
where $\Pi(\mu,\nu)$ denotes the set of probability measures on $\mathbb R^d \times \mathbb R^d$ whose first marginal is $\mu$ and whose second marginal is $\nu$.
\end{defn}

\begin{example}\label{ex:Gaussians_W2}
    Continuing as in Example \ref{ex:Gaussians_d2}, it is well-known that the $W_2$ distance between two Gaussians with the same mean is given by
    \[
    W_2^2(\mu_1,\mu_2) = \mathrm{Tr}(\Sigma_1 + \Sigma_2 - 2(\Sigma_1^{1/2}\Sigma_2 \Sigma_1^{1/2})^{1/2}).
    \]
    This formula defines a Riemannian metric on positive definite matrices, known as the Bures-Wasserstein metric, and gives an explicit example of a smooth submanifold of $\mathcal P_2(\mathbb R^d)$ equipped with the $W_2$ distance.
\end{example}

As in the previous section, we extend the $W_2$ distance to positive Radon measures with arbitrary mass and barycenter.

\begin{defn}\label{def_W2tilde}
For $\mu,\nu \in \mathcal M_2^+(\mathbb R^d)$, we define
\begin{align}\label{W2tilde}
\widetilde W_2(\mu,\nu)
:=
W_2(\bar\mu_0,\bar\nu_0)
+
|\vec{m}_\mu-\vec{m}_\nu|
+
|M_\mu-M_\nu|.
\end{align}
\end{defn}

Since $W_2$ is a distance on $\mathcal P_2(\mathbb R^d)$, the same argument as for $\widetilde d_2$ shows that $\widetilde W_2$ defines a distance on $\mathcal M_2^+(\mathbb R^d)$. Additionally, if $\mu,\nu \in \mathcal P_2(\mathbb R^d)$ and $\vec{m}_\mu=\vec{m}_\nu=\vec{0}$, then
$$
\widetilde W_2(\mu,\nu)=W_2(\mu,\nu).
$$

We will use the following H\"older-type comparison between the Fourier distance $d_2$ and the Wasserstein distance $W_2$ by Carrillo and Toscani \cite{Carrillo2007}.

\begin{prop}[Comparison between $d_2$ and $W_2$] \label{thm:d2W2Holder}
Let $d \ge 1$ and $\rho>0$. Suppose that $\mu,\nu \in \mathcal P_{2+\rho}(\mathbb R^d)$ satisfy
$$
\vec{m}_\mu=\vec{m}_\nu.
$$
Assume moreover that
$$
\langle \mu, |x|^{2+\rho} \rangle \le M,
\qquad
\langle \nu, |x|^{2+\rho} \rangle \le M,
$$
for some $M>0$. Then there exist constants $c,C>0$ and an exponent $0<q \le 1/2$, depending only on $d,\rho,M$, such that
\begin{align}\label{W2d2Holder}
c\, d_2(\mu,\nu)\le W_2(\mu,\nu)\le C\, d_2(\mu,\nu)^q.
\end{align}
\end{prop}

\begin{proof}
By Proposition 2.12 in \cite{Carrillo2007}, if $\mu$ and $\nu$ have equal barycenters, then
$$
d_2(\mu,\nu)
\le \frac12 W_2(\mu,\nu)^2
   + \min\!\left\{
      \langle \mu, |x|^2 \rangle^{1/2},
      \langle \nu, |x|^2 \rangle^{1/2}
     \right\}W_2(\mu,\nu).
$$
Set
$$
M_2:=\max\left\{\langle \mu, |x|^2 \rangle,\langle \nu, |x|^2 \rangle\right\}.
$$
Since $\mu,\nu \in \mathcal P_{2+\rho}(\mathbb R^d)$ and their $(2+\rho)$-moments are bounded by $M$, it follows that $M_2$ is bounded by a constant depending only on $\rho$ and $M$. Moreover,
$$
W_2(\mu,\nu)^2
\le 2\langle \mu, |x|^2 \rangle + 2\langle \nu, |x|^2 \rangle
\le 4M_2.
$$
Substituting this into the preceding estimate yields
$$
d_2(\mu,\nu)\le 2\sqrt{M_2}\,W_2(\mu,\nu),
$$
and therefore
$$
W_2(\mu,\nu)\ge \frac{1}{2\sqrt{M_2}}\,d_2(\mu,\nu).
$$
Since $M_2$ is controlled by $\rho$ and $M$, this gives the left inequality in \eqref{W2d2Holder}.

For the reverse inequality, we use Corollary 2.17 in \cite{Carrillo2007}, which yields
\begin{align*}
W_2(\mu,\nu)
\le
C_0
&\Bigl(
\max\{d_2(\mu,\nu),\, M_2^{\gamma_1} d_2(\mu,\nu)^{\beta_1},\, M_2^{\gamma_2} d_2(\mu,\nu)^{\beta_2},\, d_2(\mu,\nu)^{\beta_3}\}
\Bigr)^{\gamma_3}\\
&\times 
\Bigl(\max\{\langle \mu,|x|^{2+\rho}\rangle,\langle \nu,|x|^{2+\rho}\rangle\}\Bigr)^{1-\gamma_3},
\end{align*}
for some constant $C_0>0$ and exponents
$
\gamma_1,\gamma_2,\beta_1,\beta_2,\beta_3>0,
$
and $0 < \gamma_3 \le 1/2$, depending only on $d$ and $\rho$.

By the assumed $(2+\rho)$-moment bound, both $M_2$ and
$$
\max\{\langle \mu,|x|^{2+\rho}\rangle,\langle \nu,|x|^{2+\rho}\rangle\}
$$
are bounded by constants depending only on $\rho$ and $M$. Hence all moment factors can be absorbed into a new constant $C=C(d,\rho,M)$, and therefore
$$
W_2(\mu,\nu)
\le
C\,
\max\{d_2(\mu,\nu)^{\gamma_3},\, d_2(\mu,\nu)^{\beta_1\gamma_3},\, d_2(\mu,\nu)^{\beta_2\gamma_3},\, d_2(\mu,\nu)^{\beta_3\gamma_3}\}.
$$

It remains to reduce the maximum of powers to a single Hölder exponent.
Let
\[
\alpha_0:=\gamma_3,\qquad
\alpha_1:=\beta_1\gamma_3,\qquad
\alpha_2:=\beta_2\gamma_3,\qquad
\alpha_3:=\beta_3\gamma_3,
\]
and set
\[
q:=\min\{\alpha_0,\alpha_1,\alpha_2,\alpha_3\}
=\gamma_3\min\{1,\beta_1,\beta_2,\beta_3\}.
\]
Since $0<\gamma_3 \le 1/2$ and $\beta_1,\beta_2,\beta_3>0$, we have $0<q \le 1/2$.

We next observe that $d_2(\mu,\nu)$ is uniformly bounded under the present assumptions. Indeed, from the previous estimates,
\[
W_2(\mu,\nu)^2\le 4M_2
\]
and
\[
d_2(\mu,\nu)\le 2\sqrt{M_2}\,W_2(\mu,\nu).
\]
Hence
\[
d_2(\mu,\nu)\le 4M_2.
\]
Therefore, since $\alpha_j-q\ge 0$ for every $j=0,1,2,3$, we obtain
\[
\begin{aligned}
\max_{0\le j\le 3} d_2(\mu,\nu)^{\alpha_j}
&\le d_2(\mu,\nu)^q
\max_{0\le j\le 3} d_2(\mu,\nu)^{\alpha_j-q}  \\
&\le d_2(\mu,\nu)^q
\max_{0\le j\le 3} \max\{1,4M_2\}^{\alpha_j-q}.
\end{aligned}
\]
Since $M_2$ is bounded by a constant depending only on $\rho$ and $M$, the last factor is bounded by a constant depending only on $d,\rho,M$. After absorbing this factor into the constant $C=C(d,\rho,M)$, we obtain
\[
W_2(\mu,\nu)\le C d_2(\mu,\nu)^q,
\]
which is the right inequality in \eqref{W2d2Holder}.
\end{proof}

We remark that in the previous proof we utilize H\"older comparison results from \cite{Carrillo2007}. In the statements above we state that $0 \leq \gamma_3,q \leq 1/2$, which is more pessimistic than the result given in \cite{Carrillo2007}, who state the existence of $0 < \gamma_3 < 1$. As this concerns upper bounds, a smaller range is actually a weaker result, but we remark that if constants are tracked in their proof then one actually obtains $0< q \leq 1/2$. The exponent $q=1/2$ actually appears in certain contexts, as evidenced by the following example.

\begin{example}
    Continuing as in Example \ref{ex:Gaussians_d2}, we let $\Sigma_2 = 0$, corresponding to a Dirac mass at the origin. In this case, we have
    \[
    W_2^2(\mu_1,\mu_2) = \mathrm{Tr}(\Sigma_1) = \sum_{j=1}^d \lambda_j,
    \]
    whereas
    \[
    d_2(\mu_1,\mu_2) = \frac{1}{2}\|\Sigma_1\|_2 =
\frac{1}{2} \max_{1\le j\le d}\lambda_j,
    \]
    where $\lambda_1,\ldots,\lambda_d\ge 0$ denote the eigenvalues of $\Sigma_1$.
    
    Since $\frac{\mathrm{Tr}(\Sigma_1)}{d} \leq \|\Sigma_1\|_2 \leq \mathrm{Tr}(\Sigma_1)$, it follows that
    \[
    \frac{W_2^2(\mu_1,\mu_2)}{2d}
    \leq d_2(\mu_1,\mu_2)
    \leq \frac{W_2^2(\mu_1,\mu_2)}{2}.
    \]
    Thus, in this example, $W_2(\mu_1,\mu_2)$ is comparable to
    $d_2(\mu_1,\mu_2)^{1/2}$, which is compatible with the previous proposition with $q=1/2$.
\end{example}

\begin{remark}[Sharpness of the exponent]
The exponent in \eqref{W2d2Holder} cannot, in general, be taken larger than
$1/2$. Indeed, let
\[
\mu_\varepsilon=\mathcal N(0,\varepsilon I_d),
\qquad
\nu=\delta_0,
\qquad
0<\varepsilon<1.
\]
Then $\vec m_{\mu_\varepsilon}=\vec m_\nu=0$. Also,
\[
\int_{\mathbb R^d}|x|^{2+\rho}\,d\mu_\varepsilon(x)
=
\frac{1}{(2\pi\varepsilon)^{d/2}}
\int_{\mathbb R^d}
|x|^{2+\rho}
\exp\left(-\frac{|x|^2}{2\varepsilon}\right)\,dx.
\]
By the change of variables $x=\sqrt{\varepsilon}\,y$, this becomes
\[
\int_{\mathbb R^d}|x|^{2+\rho}\,d\mu_\varepsilon(x)
=
\varepsilon^{(2+\rho)/2}
\frac{1}{(2\pi)^{d/2}}
\int_{\mathbb R^d}
|y|^{2+\rho}e^{-|y|^2/2}\,dy
=
(2\varepsilon)^{(2+\rho)/2}
\frac{\Gamma\left(\frac{d+\rho+2}{2}\right)}{\Gamma(d/2)},
\]
where $\Gamma$ denotes the Gamma function. In addition,
\[
\int_{\mathbb R^d}|x|^{2+\rho}\,d\nu(x)=0.
\]

Hence, $\mu_\varepsilon,\nu\in\mathcal P_{2+\rho}(\mathbb R^d)$. Moreover, for any fixed $M>0$, the moment assumptions are satisfied for all sufficiently small $\varepsilon$.

On the other hand, by Examples \ref{ex:Gaussians_d2} and \ref{ex:Gaussians_W2},
\[
d_2(\mu_\varepsilon,\nu)
=
\frac12\|\varepsilon I_d\|_2
=
\frac{\varepsilon}{2},
\]
while
\[
W_2^2(\mu_\varepsilon,\nu)
=
\operatorname{Tr}(\varepsilon I_d)
=
\varepsilon d.
\]
Therefore
\[
W_2(\mu_\varepsilon,\nu)
=
\sqrt{2d}\, d_2(\mu_\varepsilon,\nu)^{1/2}.
\]
If an estimate of the form
\[
W_2(\mu,\nu)\le C d_2(\mu,\nu)^q
\]
held uniformly under the assumptions of Proposition \ref{thm:d2W2Holder} for
some $q>1/2$, then applying it to the pair
$(\mu_\varepsilon,\nu)$ would give
\[
\sqrt{2d}
\le
C d_2(\mu_\varepsilon,\nu)^{q-1/2}.
\]
Since $d_2(\mu_\varepsilon,\nu)\to0$ as $\varepsilon\to0$, the right-hand side
tends to zero, a contradiction. Thus the exponent $1/2$ is sharp in this
generality.
\end{remark}

In view of Proposition \ref{p:PdistEqWass}, the preceding result immediately yields an estimate for the $k$-plane transform in terms of the Wasserstein distance.
\begin{corollary}\label{corr:D_Hequiv_W2}
    Let $d \ge 1$ and $\rho>0$. Suppose that $\mu,\nu \in \mathcal P_{2+\rho}(\mathbb R^d)$ satisfy $\vec{m}_\mu=\vec{m}_\nu = 0.$ Assume moreover that
$$
\langle \mu, |x|^{2+\rho} \rangle \le M,
\qquad
\langle \nu, |x|^{2+\rho} \rangle \le M,
$$
for some $M>0$. Then there exist constants $c,C>0$ and an exponent $0<q \le 1/2$, depending only on $d,\rho,M$, such that
\begin{align}\label{eq:D_Hequiv_W2}
c\, W_2(\mu,\nu)^{1/q}\le D(P\mu,P\nu)\le C\, W_2(\mu,\nu).
\end{align}
\end{corollary}

Proposition \ref{thm:d2W2Holder} extends to the generalized distances $\widetilde d_2$ and $\widetilde W_2$. Observe that
$$
\langle \bar\mu_0, |x|^{2+\rho}\rangle
=
\frac1{M_\mu}\int_{\mathbb R^d}|x-\vec{m}_\mu|^{2+\rho}\,d\mu(x),
$$
and likewise for $\nu$. Thus the uniform moment assumption on $\bar\mu_0$ and $\bar\nu_0$ may be stated directly in terms of the centered $(2+\rho)$-moments of $\mu$ and $\nu$, normalized by their masses.

\begin{prop}\label{p:td2tW2Holder}
Let $\rho>0$, and let $\mu,\nu \in \mathcal M_{2+\rho}^+(\mathbb R^d)$. Assume that there exists $M>0$ such that
\begin{align*}
\frac1{M_\mu}\int_{\mathbb R^d}|x-\vec{m}_\mu|^{2+\rho}\,d\mu(x)\le M,
\qquad
\frac1{M_\nu}\int_{\mathbb R^d}|x-\vec{m}_\nu|^{2+\rho}\,d\nu(x)\le M.
\end{align*}
Then there exist constants $c,C>0$ and an exponent $0<q \le 1/2$, depending only on $d,\rho,M$, such that
\begin{align}\label{td2tW2Holder}
c\, \widetilde d_2(\mu,\nu)
\le
\widetilde W_2(\mu,\nu)
\le
C\bigl(\widetilde d_2(\mu,\nu)^q+\widetilde d_2(\mu,\nu)\bigr).
\end{align}
\end{prop}

\begin{proof}
Since
\begin{align*}
\bar\mu_0=\frac1{M_\mu}T_{\vec{m}_\mu}\mu,
\qquad
\bar\nu_0=\frac1{M_\nu}T_{\vec{m}_\nu}\nu,
\end{align*}
the assumptions imply that $\bar\mu_0,\bar\nu_0\in \mathcal P_{2+\rho}(\mathbb R^d)$, both have zero barycenter, and satisfy
\begin{align*}
\langle \bar\mu_0,|x|^{2+\rho}\rangle\le M,
\qquad
\langle \bar\nu_0,|x|^{2+\rho}\rangle\le M.
\end{align*}
Hence Proposition \ref{thm:d2W2Holder} applies to the pair $(\bar\mu_0,\bar\nu_0)$. Therefore there exist constants $c_0,C_0>0$ and an exponent $0<q \le 1/2$, depending only on $d,\rho,M$, such that
\begin{align*}
c_0\, d_2(\bar\mu_0,\bar\nu_0)
\le
W_2(\bar\mu_0,\bar\nu_0)
\le
C_0\, d_2(\bar\mu_0,\bar\nu_0)^q.
\end{align*}

For the lower bound, using the definitions of $\td_2$ and $\tW_2$, we obtain
\begin{align*}
\widetilde W_2(\mu,\nu)
&=
W_2(\bar\mu_0,\bar\nu_0)
+
|\vec{m}_\mu-\vec{m}_\nu|
+
|M_\mu-M_\nu| \\
&\ge
c_0\,d_2(\bar\mu_0,\bar\nu_0)
+
|\vec{m}_\mu-\vec{m}_\nu|
+
|M_\mu-M_\nu| \\
&\ge
\min\{c_0,1\}
\left(
d_2(\bar\mu_0,\bar\nu_0)
+
|\vec{m}_\mu-\vec{m}_\nu|
+
|M_\mu-M_\nu|
\right) \\
&=
\min\{c_0,1\}\,\widetilde d_2(\mu,\nu).
\end{align*}
Thus the lower bound holds with $c:=\min\{c_0,1\}$.

For the upper bound, again by the definitions of $\td_2$ and $\tW_2$,
\begin{align*}
\widetilde W_2(\mu,\nu)
&=
W_2(\bar\mu_0,\bar\nu_0)
+
|\vec{m}_\mu-\vec{m}_\nu|
+
|M_\mu-M_\nu| \\
&\le
C_0\,d_2(\bar\mu_0,\bar\nu_0)^q
+
|\vec{m}_\mu-\vec{m}_\nu|
+
|M_\mu-M_\nu|.
\end{align*}
Since
\begin{align*}
d_2(\bar\mu_0,\bar\nu_0)\le \widetilde d_2(\mu,\nu),
\end{align*}
we have
\begin{align*}
d_2(\bar\mu_0,\bar\nu_0)^q
\le
\widetilde d_2(\mu,\nu)^q.
\end{align*}
Moreover,
\begin{align*}
|\vec{m}_\mu-\vec{m}_\nu|
+
|M_\mu-M_\nu|
\le
\widetilde d_2(\mu,\nu).
\end{align*}
Combining these estimates gives
\begin{align*}
\widetilde W_2(\mu,\nu)
&\le
C_0\,\widetilde d_2(\mu,\nu)^q
+
\widetilde d_2(\mu,\nu) 
\le
C\bigl(\widetilde d_2(\mu,\nu)^q+\widetilde d_2(\mu,\nu)\bigr)
\end{align*}
for some constant $C>0$ depending only on $d,\rho,M$. This proves \eqref{td2tW2Holder}.
\end{proof}

Combining Proposition \ref{p:td2tW2Holder} with the bi-Lipschitz estimate of Theorem \ref{t:kplaneEq_d2_general}, we obtain the desired H\"older estimates for the $k$-plane transform with respect to the generalized Wasserstein distance.

\begin{thm}\label{thm:kplaneW2Holder}
Let $\rho>0$, and let $\mu,\nu \in \mathcal M_{2+\rho}^+(\mathbb R^d)$. Assume that there exists $M>0$ such that
\begin{align*}
\frac1{M_\mu}\int_{\mathbb R^d}|x-\vec{m}_\mu|^{2+\rho}\,d\mu(x)\le M,
\qquad
\frac1{M_\nu}\int_{\mathbb R^d}|x-\vec{m}_\nu|^{2+\rho}\,d\nu(x)\le M.
\end{align*}
Then there exist constants $c,C>0$ and an exponent $0<q \le 1/2$, depending only on $d,\rho,M$, such that
\begin{align}\label{eq:kplaneW2Holder}
c\, \min\bigl\{\widetilde W_2(\mu,\nu)^{1/q},\widetilde W_2(\mu,\nu)\bigr\}
\le
D(P\mu,P\nu)
\le
C\, \widetilde W_2(\mu,\nu).
\end{align}
\end{thm}

\begin{proof}
By Theorem \ref{t:kplaneEq_d2_general}, we have
\begin{align*}
\frac{1}{2}\,\widetilde d_2(\mu,\nu)
\le
D(P\mu,P\nu)
\le
\widetilde d_2(\mu,\nu).
\end{align*}
Moreover, by Proposition \ref{p:td2tW2Holder}, there exist constants $c_0,C_0>0$ and an exponent $0<q \le 1/2$, depending only on $d,\rho,M$, such that
\begin{align*}
c_0\, \widetilde d_2(\mu,\nu)
\le
\widetilde W_2(\mu,\nu)
\le
C_0\bigl(\widetilde d_2(\mu,\nu)^q+\widetilde d_2(\mu,\nu)\bigr).
\end{align*}
The upper bound in \eqref{eq:kplaneW2Holder} follows immediately:
\begin{align*}
D(P\mu,P\nu)
\le
\widetilde d_2(\mu,\nu)
\le
c_0^{-1}\widetilde W_2(\mu,\nu).
\end{align*}

It remains to prove the lower bound. We distinguish two cases. If
$0\le \widetilde d_2(\mu,\nu)\le 1$, then, since $0<q \le 1/2$,
\begin{align*}
\widetilde d_2(\mu,\nu)
\le
\widetilde d_2(\mu,\nu)^q.
\end{align*}
Hence
\begin{align*}
\widetilde W_2(\mu,\nu)
\le
2C_0\,\widetilde d_2(\mu,\nu)^q,
\end{align*}
and therefore
\begin{align*}
\widetilde d_2(\mu,\nu)
\ge
(2C_0)^{-1/q}\,\widetilde W_2(\mu,\nu)^{1/q}.
\end{align*}
On the other hand, if $\widetilde d_2(\mu,\nu)\ge 1$, then
\begin{align*}
\widetilde d_2(\mu,\nu)^q
\le
\widetilde d_2(\mu,\nu),
\end{align*}
and hence
\begin{align*}
\widetilde W_2(\mu,\nu)
\le
2C_0\,\widetilde d_2(\mu,\nu).
\end{align*}
Thus
\begin{align*}
\widetilde d_2(\mu,\nu)
\ge
(2C_0)^{-1}\,\widetilde W_2(\mu,\nu).
\end{align*}
Combining the two cases, we obtain
\begin{align*}
\widetilde d_2(\mu,\nu)
\ge
c_1\,\min\bigl\{
\widetilde W_2(\mu,\nu)^{1/q},
\widetilde W_2(\mu,\nu)
\bigr\}
\end{align*}
for some constant $c_1>0$ depending only on $d,\rho,M$. Therefore,
\begin{align*}
D(P\mu,P\nu)
\ge
\frac12\,\widetilde d_2(\mu,\nu)
\ge
\frac{c_1}{2}\,
\min\bigl\{
\widetilde W_2(\mu,\nu)^{1/q},
\widetilde W_2(\mu,\nu)
\bigr\}.
\end{align*}
Absorbing constants into $c$ and $C$ gives \eqref{eq:kplaneW2Holder}.
\end{proof}

\section{H\"older-type comparison between $W_2$ and max-sliced $W_2$ distances}

In this section, we use the H\"older-type stability estimates obtained in the previous section to compare the $W_2$ distance with its max-sliced analogue.
\begin{defn}\label{def:MSW2}
For $\mu,\nu \in \mathcal P_2(\mathbb R^d)$, the max-sliced (or projection robust) $W_2$ distance is defined by
\begin{align}
MSW_2(\mu,\nu)
:=
\sup_{\alpha \in G_{k,d}} W_2(P_\alpha\mu,P_\alpha\nu).
\end{align}
For $\mu,\nu \in \mathcal M_2^+(\mathbb R^d)$, the generalized max-sliced $W_2$ distance is defined by
\begin{align}
MS\tW_2(\mu,\nu)
:=
\sup_{\alpha \in G_{k,d}} \tW_2(P_\alpha\mu,P_\alpha\nu).
\end{align}
\end{defn}

Projection-based Wasserstein distances are widely used in optimal transport
and machine learning as computationally cheaper alternatives to full
Wasserstein distances. In the $W_1$ case, strong and sharp comparison
results are known in several forms; see, for example,
\cite{BayraktarGuo2021,Carlier2025}. In the $W_2$ setting, Paty and Cuturi
\cite{PatyCuturi2019} proved a strong equivalence between $W_2$ and the
subspace-robust distance $S_k$, a min--max relaxation of a projection-based
transport problem. Here, we prove a two-sided H\"older-type comparison between $W_2$ and the max-sliced $W_2$ distance.

\begin{thm}\label{maxSW-W}
Let $\rho>0$, let $0\le k<d$, and let
$\mu,\nu \in \mathcal M_{2+\rho}^+(\mathbb R^d)$. Assume that there exists
$M>0$ such that
\begin{align*}
\frac1{M_\mu}\int_{\mathbb R^d}|x-\vec{m}_\mu|^{2+\rho}\,d\mu(x)\le M,
\qquad
\frac1{M_\nu}\int_{\mathbb R^d}|x-\vec{m}_\nu|^{2+\rho}\,d\nu(x)\le M.
\end{align*}
Then there exist a constant $c>0$ and an exponent $0<q \le 1/2$, depending only on
$k,d,\rho,M$, such that
\begin{align}\label{W2equivMSW2}
c\,\min\bigl\{\tW_2(\mu,\nu)^{1/q},\tW_2(\mu,\nu)\bigr\}
\le
MS\tW_2(\mu,\nu)
\le
\tW_2(\mu,\nu).
\end{align}
In particular, if $\mu,\nu \in \mathcal P_{2+\rho}(\mathbb R^d)$ and
$\vec{m}_\mu=\vec{m}_\nu=\vec{0}$, then, possibly after changing the
constant $c>0$ and the exponent $0<q \le 1/2$ while preserving their dependence
only on $k,d,\rho,M$,
\begin{align}\label{eq:MSW2-W2-prob}
c\,W_2(\mu,\nu)^{1/q}
\le
MSW_2(\mu,\nu)
\le
W_2(\mu,\nu).
\end{align}
\end{thm}

\begin{proof}
Fix $\alpha\in G_{k,d}$. If $\lambda\in\mathcal M_{2+\rho}^+(\mathbb R^d)$,
then by Proposition \ref{p:kplaneProperties}, $P_\alpha\lambda=(\pi_{\alpha^\perp})_\#\lambda$ satisfies
\begin{align}
M_{P_\alpha\lambda}=M_\lambda,
\qquad
\vec{m}_{P_\alpha\lambda}
=
\pi_{\alpha^\perp}\vec{m}_\lambda.
\label{eq:Palpha-mass-mean}
\end{align}
Moreover, since $\pi_{\alpha^\perp}$ is nonexpansive,
\begin{align}
\frac{1}{M_{P_\alpha\lambda}}
\int_{\alpha^\perp}
|y-\vec{m}_{P_\alpha\lambda}|^{2+\rho}\,dP_\alpha\lambda(y)
&=
\frac{1}{M_\lambda}
\int_{\mathbb R^d}
\left|\pi_{\alpha^\perp}(x-\vec{m}_\lambda)\right|^{2+\rho}
\,d\lambda(x)
\notag \\
&\le
\frac{1}{M_\lambda}
\int_{\mathbb R^d}
|x-\vec{m}_\lambda|^{2+\rho}\,d\lambda(x).
\label{eq:Palpha-moment-bound}
\end{align}

We first prove \eqref{eq:MSW2-W2-prob}. Assume that
$\mu,\nu\in\mathcal P_{2+\rho}(\mathbb R^d)$,
$\vec{m}_\mu=\vec{m}_\nu=\vec{0}$, and
\[
\int_{\mathbb R^d}|x|^{2+\rho}\,d\mu(x)\le M,
\qquad
\int_{\mathbb R^d}|x|^{2+\rho}\,d\nu(x)\le M.
\]

By \eqref{eq:Palpha-mass-mean} and \eqref{eq:Palpha-moment-bound},
for every $\alpha\in G_{k,d}$, the projected measures
$P_\alpha\mu$ and $P_\alpha\nu$ are centered probability measures on
$\alpha^\perp$ satisfying the same $(2+\rho)$-moment bound. Therefore
Proposition \ref{thm:d2W2Holder} applies uniformly in $\alpha$ to the pair $(P_\alpha\mu,P_\alpha\nu)$ and gives
a constant $c_0>0$, depending only on $d-k,\rho,M$, such that
\[
c_0\,d_2(P_\alpha\mu,P_\alpha\nu)
\le
W_2(P_\alpha\mu,P_\alpha\nu),
\]
for every $\alpha\in G_{k,d}$. Taking the supremum over $\alpha$ gives
\[
c_0D(P\mu,P\nu) = c_0\sup_{\alpha\in G_{k,d}} d_2(P_\alpha\mu,P_\alpha\nu) \le \sup_{\alpha\in G_{k,d}} W_2(P_\alpha\mu,P_\alpha\nu) =MSW_2(\mu,\nu).
\]
By Corollary \ref{corr:D_Hequiv_W2}, there exist $c_1>0$ and
$0<q_1 \le 1/2$, depending only on $k,d,\rho,M$, such that
\[
c_1W_2(\mu,\nu)^{1/q_1}
\le
D(P\mu,P\nu).
\]
Combining the last two estimates yields
\[
c\,W_2(\mu,\nu)^{1/q_1}
\le
MSW_2(\mu,\nu).
\]

Next, we prove the upper bound. If $\eta,\zeta\in\mathcal P_2(\mathbb R^d)$, then for any $\gamma\in\Pi(\eta,\zeta)$,
\[
(\pi_{\alpha^\perp},\pi_{\alpha^\perp})_\#\gamma
\in
\Pi(P_\alpha\eta,P_\alpha\zeta).
\]
Thus
\begin{align*}
W_2(P_\alpha\eta,P_\alpha\zeta)^2
\le
\int_{\mathbb R^d\times\mathbb R^d}
\left|\pi_{\alpha^\perp}(x-y)\right|^2\,d\gamma(x,y) 
\le
\int_{\mathbb R^d\times\mathbb R^d}
|x-y|^2\,d\gamma(x,y).
\end{align*}
Taking the infimum over $\gamma\in\Pi(\eta,\zeta)$ gives
\begin{align}
W_2(P_\alpha\eta,P_\alpha\zeta)
\le
W_2(\eta,\zeta),
\label{eq:Palpha-W2-contraction}
\end{align}
and taking the supremum over $\alpha$ yields
\[
MSW_2(\mu,\nu)
=
\sup_{\alpha\in G_{k,d}}W_2(P_\alpha\mu,P_\alpha\nu)
\le
W_2(\mu,\nu).
\]
Renaming $q_1$ as $q$ proves \eqref{eq:MSW2-W2-prob}.

We now prove the lower bound in \eqref{W2equivMSW2}. Let
$\mu,\nu\in\mathcal M_{2+\rho}^+(\mathbb R^d)$ satisfy the stated
normalized moment bounds. By \eqref{eq:Palpha-moment-bound}, the same
normalized moment bounds hold uniformly for $P_\alpha\mu$ and
$P_\alpha\nu$. Hence Proposition \ref{p:td2tW2Holder} applies uniformly in
$\alpha$ and gives a constant $c_2>0$, depending only on $d-k,\rho,M$,
such that
\[
c_2\,\td_2(P_\alpha\mu,P_\alpha\nu)
\le
\tW_2(P_\alpha\mu,P_\alpha\nu),
\]
for every $\alpha\in G_{k,d}$. Taking the supremum over $\alpha$ gives
\[
c_2D(P\mu,P\nu)\le MS\tW_2(\mu,\nu).
\]
By Theorem \ref{thm:kplaneW2Holder}, there exist $c_3>0$ and
$0<q_2 \le 1/2$, depending only on $k,d,\rho,M$, such that
\[
c_3\min\left\{
\tW_2(\mu,\nu)^{1/q_2},\tW_2(\mu,\nu)
\right\}
\le
D(P\mu,P\nu).
\]
Therefore,
\[
c\,\min\left\{
\tW_2(\mu,\nu)^{1/q_2},\tW_2(\mu,\nu)
\right\}
\le
MS\tW_2(\mu,\nu).
\]

It remains to prove the upper bound in \eqref{W2equivMSW2}. By \eqref{eq:Palpha-mass-mean},
the mass term in $\tW_2$ is unchanged under $P_\alpha$. The barycenter term
is nonincreasing because
\begin{align}
\left|\vec{m}_{P_\alpha\mu}-\vec{m}_{P_\alpha\nu}\right|
&=
\left|\pi_{\alpha^\perp}(\vec{m}_\mu-\vec{m}_\nu)\right|
\le
\left|\vec{m}_\mu-\vec{m}_\nu\right|.
\label{eq:Palpha-barycenter-contraction}
\end{align}
Finally, the centered normalized Wasserstein term is nonincreasing by
\eqref{eq:Palpha-W2-contraction}. Hence
\[
\tW_2(P_\alpha\mu,P_\alpha\nu)
\le
\tW_2(\mu,\nu).
\]
Taking the supremum over $\alpha\in G_{k,d}$ gives
\[
MS\tW_2(\mu,\nu)
=
\sup_{\alpha\in G_{k,d}}\tW_2(P_\alpha\mu,P_\alpha\nu)
\le
\tW_2(\mu,\nu).
\]
Renaming $q_2$ as $q$ proves \eqref{W2equivMSW2}.
\end{proof}

\section{Bi-Lipschitz $W_2$ stability for compactly supported densities via Sobolev estimates}

In the previous sections, we established stability estimates for the $k$-plane transform in terms of the Fourier and the Wasserstein distances on measures. We now specialize to absolutely continuous probability measures with bounded and compactly supported densities. In this setting, the Wasserstein distance is quantitatively comparable to the homogeneous $H^{-1}$ norm, while the $k$-plane transform admits a Sobolev stability estimate in the $H^{-1}$ scale. Combining these facts yields a direct stability estimate for the $k$-plane transform in terms of the $2$-Wasserstein distance.

We first recall the homogeneous Sobolev seminorm
\begin{align}\label{Sobolevnorm}
    \|f\|_{\dot H^s(\mathbb R^d)}
    :=
    \left(
    \int_{\mathbb R^d} |\widehat f(\xi)|^2 |\xi|^{2s}\, d\xi
    \right)^{1/2}, \quad s \in \R.
\end{align}
In particular, for $f \in \dot H^{-1}(\mathbb R^d)$, the seminorm $\|f\|_{\dot H^{-1}(\mathbb R^d)}$ is the dual norm of $\dot H^1(\mathbb R^d)$, namely
$$
\|f\|_{\dot H^{-1}(\mathbb R^d)}
=
\sup \left\{
|\langle f,\phi\rangle| : \phi \in \dot H^1(\mathbb R^d),\ \|\phi\|_{\dot H^1(\mathbb R^d)} \le 1
\right\}.
$$

The following equivalence is likely standard, but since we were unable to locate a convenient reference in the precise form needed below, we include a proof.

\begin{prop}[Equivalence of the $H^{-1}$ norm and the $\dot H^{-1}$ seminorm]\label{Sobolev-1Equivalence}
Let $d\ge 2$, and suppose that $f \in H^{-1}(\mathbb R^d)$ satisfy
$$
\operatorname{supp}(f)\subseteq \overline{\Omega},
$$
for some bounded, open set $\Omega \subseteq \mathbb R^d$. If $d=2$, assume in addition that $f$ has zero mean, i.e.,
$$
\widehat f(0)=\int_{\mathbb R^2} f(x)\,dx=0.
$$
Then
\begin{align}
  \|f\|_{H^{-1}(\mathbb R^d)} \simeq \|f\|_{\dot H^{-1}(\mathbb R^d)}.
\end{align}
\end{prop}

\begin{proof}
One direction is immediate:
\begin{align*}
\|f\|^2_{H^{-1}(\mathbb R^d)}
&=
\int_{\mathbb R^d} (1+|\xi|^2)^{-1} |\widehat f(\xi)|^2\,d\xi \\
&\le
\int_{\mathbb R^d} |\xi|^{-2} |\widehat f(\xi)|^2\,d\xi
=
\|f\|^2_{\dot H^{-1}(\mathbb R^d)}.
\end{align*}
For the converse estimate, we split the homogeneous norm into high and low frequencies:
$$
\|f\|^2_{\dot H^{-1}(\mathbb R^d)}
=
\int_{|\xi|\ge 1} |\xi|^{-2} |\widehat f(\xi)|^2\,d\xi
+
\int_{|\xi|\le 1} |\xi|^{-2} |\widehat f(\xi)|^2\,d\xi.
$$
For $|\xi|\ge 1$, we have
$$
|\xi|^2 \ge \frac12(1+|\xi|^2),
$$
and therefore
$$
\int_{|\xi|\ge 1} |\xi|^{-2} |\widehat f(\xi)|^2\,d\xi
\le
2\int_{|\xi|\ge 1} (1+|\xi|^2)^{-1} |\widehat f(\xi)|^2\,d\xi
\le
2\|f\|^2_{H^{-1}(\mathbb R^d)}.
$$

It remains to estimate the low-frequency part.
If $d>2$, then
$$
\int_{|\xi|\le 1} |\xi|^{-2} |\widehat f(\xi)|^2\,d\xi
\le
\left(\int_{|\xi|\le 1} |\xi|^{-2}\,d\xi\right)
\sup_{|\xi|\le 1} |\widehat f(\xi)|^2
=
\frac{|\mathbb S^{d-1}|}{d-2}\sup_{|\xi|\le 1} |\widehat f(\xi)|^2.
$$
Thus it suffices to bound $\sup_{|\xi|\le 1} |\widehat f(\xi)|$ in terms of $\|f\|_{H^{-1}}$.

Choose $\chi \in C_c^\infty(\mathbb R^d)$ such that $\chi=1$ on $\Omega$, and define
$$
\chi_\xi(x):=e^{-ix\cdot \xi}\chi(x).
$$
Since $\operatorname{supp}(f)\subseteq \Omega$, we have
$$
\widehat f(\xi)
=
\int_{\mathbb R^d} f(x)e^{-ix\cdot \xi}\,dx
=
\int_{\mathbb R^d} f(x)\chi(x)e^{-ix\cdot \xi}\,dx
=
\langle f,\chi_\xi\rangle.
$$
By Plancherel's theorem and the Cauchy--Schwarz inequality,
\begin{align}\label{PlancherelCS-new}
|\widehat f(\xi)|^2
&=
\left|
\int_{\mathbb R^d}
\widehat f(\eta)\,\widehat{\chi_\xi}(\eta)\,d\eta
\right|^2 \\
&\le
\left(
\int_{\mathbb R^d} (1+|\eta|^2)^{-1} |\widehat f(\eta)|^2\,d\eta
\right)
\left(
\int_{\mathbb R^d} (1+|\eta|^2) |\widehat{\chi_\xi}(\eta)|^2\,d\eta
\right)\nonumber \\
&=
\|f\|_{H^{-1}(\mathbb R^d)}^2 \|\chi_\xi\|_{H^1(\mathbb R^d)}^2.\nonumber
\end{align}
Moreover,
$$
\|\chi_\xi\|_{H^1(\mathbb R^d)}^2
=
\int_{\mathbb R^d} (1+|\eta|^2) |\widehat\chi(\eta+\xi)|^2\,d\eta,
$$
which is continuous in $\xi$, hence bounded for $|\xi|\le 1$. Therefore,
$$
\sup_{|\xi|\le 1} |\widehat f(\xi)|^2
\le
C\,\|f\|^2_{H^{-1}(\mathbb R^d)},
$$
which proves the claim for $d>2$.

If $d=2$, we write
$$
\int_{|\xi|\le 1} |\xi|^{-2} |\widehat f(\xi)|^2\,d\xi
\le
2\pi \sup_{|\xi|\le 1} \frac{|\widehat f(\xi)|^2}{|\xi|}.
$$
Let
$$
\chi_\xi(x):=\frac{e^{-ix\cdot \xi}-1}{|\xi|^{1/2}}\chi(x),
$$
where $\chi$ is the cutoff chosen above. Since $\widehat f(0)=\int_{\R^2} f(x)\,dx=0$, we obtain
\begin{align*}
\frac{\widehat f(\xi)}{|\xi|^{1/2}}
&=
\frac{1}{|\xi|^{1/2}}
\left(
\int_{\mathbb R^2} f(x)e^{-ix\cdot \xi}\,dx
-
\int_{\mathbb R^2} f(x)\,dx
\right) \\
&=
\int_{\mathbb R^2} f(x)\,\frac{e^{-ix\cdot \xi}-1}{|\xi|^{1/2}}\,\chi(x)\,dx
=
\langle f,\chi_\xi\rangle.
\end{align*}
Applying the same argument as in \eqref{PlancherelCS-new}, we get
$$
|\xi|^{-1} |\widehat f(\xi)|^2
\le
\|\chi_\xi\|_{H^1(\mathbb R^2)}^2 \|f\|_{H^{-1}(\mathbb R^2)}^2.
$$
It remains to show that $\|\chi_\xi\|_{H^1}$ is bounded uniformly for $|\xi|\le 1$. Since
$$
\widehat{\chi_\xi}(\eta)
=
\frac{\widehat\chi(\eta+\xi)-\widehat\chi(\eta)}{|\xi|^{1/2}},
$$
the mean value theorem gives, for $|\xi|\le 1$,
$$
|\widehat{\chi_\xi}(\eta)|^2
\le
|\xi|\,
\sup_{|\zeta|\le 1} |\nabla \widehat\chi(\eta+\zeta)|^2.
$$
Therefore,
\begin{align*}
\|\chi_\xi\|_{H^1(\mathbb R^2)}^2
&=
\int_{\mathbb R^2} (1+|\eta|^2)|\widehat{\chi_\xi}(\eta)|^2\,d\eta \\
&\le
|\xi|\int_{\mathbb R^2} (1+|\eta|^2)\sup_{|\zeta|\le 1} |\nabla \widehat\chi(\eta+\zeta)|^2\,d\eta.
\end{align*}
Since $\chi \in C_c^\infty(\mathbb R^2)$, the Fourier transform $\widehat\chi$ and all its derivatives decay faster than any polynomial, so the integral on the right-hand side is finite. Hence $\|\chi_\xi\|_{H^1(\mathbb R^2)}$ is uniformly bounded for $|\xi|\le 1$, and the proof is complete.
\end{proof}

We now combine this Sobolev equivalence with the $W_2$–$\dot H^{-1}$ comparison of Peyré and the Sobolev stability estimate for the $k$-plane transform. In what follows, for an absolutely continuous measure $\mu=f\,dx$, we write $Pf:=P(f\,dx)$.

\begin{thm}\label{thm:kplaneW2densities}
Let $d\ge 2$, and let $\Omega\subset \mathbb R^d$ be a bounded connected Lipschitz domain. Let
$$
\mu=f(x)\,dx,
\qquad
\nu=g(x)\,dx
$$
be probability measures supported in $\overline{\Omega}$. Assume that $f$ and $g$ vanish outside $\Omega$ and that there exist constants $0<b\le B$ such that
$$
b\le f(x),\, g(x)\le B
\qquad\text{for a.e. }x\in \Omega.
$$
Then there exist constants $c,C>0$, depending only on $b,B,d,k$ and $\Omega$, such that
\begin{align}\label{eq:kplaneW2densities}
cW_2(\mu,\nu) \le \|Pf-Pg\|_{H^{k/2-1}(\mathcal G_{k,d})}
\le C W_2(\mu,\nu).
\end{align}
\end{thm}

\begin{proof}
Let $h:=f-g$, extended by zero outside $\Omega$. Since $\mu$ and $\nu$ are probability measures, $h$ has zero mean.

By Theorem 2.5 in \cite{Peyre2018}, applied on $\mathbb R^d$, and using the global upper bounds
$$
\mu,\nu\le B\,dx,
$$
we obtain
$$
\|h\|_{\dot H^{-1}(\mathbb R^d)}
\le
C_0 W_2(\mu,\nu),
$$
where $C_0>0$ depends only on $B$.

Conversely, applying Corollary 2.3 in \cite{Peyre2018} on $\Omega$ with its intrinsic distance and volume measure, and using the lower bound $f,g\ge b$, we obtain
$$
W_2(\mu,\nu)
\le
W_{2,\Omega}(\mu,\nu)
\le
C_1\|h\|_{\dot H^{-1}(\Omega)},
$$
where $W_{2,\Omega}$ denotes the Wasserstein distance associated with the intrinsic distance on $\Omega$.

Moreover, because $\Omega$ is a bounded Lipschitz domain and $h$ has zero mean, the standard extension theorem and Poincar\'e inequality imply
$$
\|h\|_{\dot H^{-1}(\Omega)}
\le
C_\Omega \|h\|_{\dot H^{-1}(\mathbb R^d)}.
$$
Therefore,
$$
W_2(\mu,\nu)
\le
C_2\|h\|_{\dot H^{-1}(\mathbb R^d)}.
$$
Combining the two estimates gives
$$
\|f-g\|_{\dot H^{-1}(\mathbb R^d)}
\simeq
W_2(\mu,\nu),
$$
with constants depending only on $b,B$ and $\Omega$.

Since $h=f-g$ is compactly supported, Theorem \ref{kplaneSobolevEstimate} applies with $s=-1$ and gives
$$
\|Pf-Pg\|_{H^{k/2-1}(\mathcal G_{k,d})}
\simeq
\|f-g\|_{H^{-1}(\mathbb R^d)}.
$$
Finally, Proposition \ref{Sobolev-1Equivalence} gives
$$
\|f-g\|_{H^{-1}(\mathbb R^d)}
\simeq
\|f-g\|_{\dot H^{-1}(\mathbb R^d)}.
$$
Combining these estimates proves \eqref{eq:kplaneW2densities}.
\end{proof}

\section{Conclusion}
In this paper, we developed a stability theory for the $k$-plane transform on finite positive Radon measures. We introduced a metric on $k$-plane transform data that is bi-Lipschitz comparable to a generalized Fourier metric augmenting the centered normalized Fourier distance with barycenter and total mass differences. This provides a measure-valued analogue of Sobolev stability estimates for the $k$-plane transform and shows that distances between $k$-plane data directly control, and are controlled by, Fourier-based distances between the underlying measures.

We then related these Fourier-based estimates to Wasserstein distances. Building on the comparison between Fourier and $2$-Wasserstein distances for probability measures \cite{Carrillo2007}, we proved corresponding estimates for finite positive Radon measures in terms of generalized Fourier and Wasserstein distances. This yielded H\"older-type stability estimates for the $k$-plane transform in generalized Wasserstein-type metrics. Finally, we established a two-sided H\"older-type comparison between the $2$-Wasserstein distance and its max-sliced analogue, first for centered probability measures and then, after adding separate mass and barycenter terms, for finite positive Radon measures.

Finally, in the absolutely continuous case, we used Sobolev stability estimates for the $k$-plane transform together with Wasserstein-Sobolev comparisons under uniform density bounds \cite{Peyre2018} to obtain a bi-Lipschitz estimate in the $2$-Wasserstein metric for compactly supported probability densities bounded above and below. Altogether, these results show that the $k$-plane transform admits a quantitative stability theory for measures in both Fourier and Wasserstein metrics.

Several directions remain open. One natural question is whether the generalized comparisons obtained here can be sharpened. Other directions include extending the analysis to broader classes of measures under weaker moment or boundedness assumptions, and investigating analogous stability questions for other integral transforms arising in inverse problems and imaging.

\section{Appendix}\label{s:appendix}
\begin{proof}[\textbf{Proof of ${\bf \psi \in C_0(\cG_{k,d})}$ implies $\bf{P^*\psi  \in C_0(\R^d)}$:}]
Since $\psi$ is bounded on $\cG_{k,d}$ and $d\sigma$ is a probability measure on $G_{k,d}$, for all $x \in \mathbb{R}^d$, we have
$$
|(P^*\psi)(x)|
\le
\int_{G_{k,d}} \bigl|\psi\bigl(\alpha,\pi_{\alpha^\perp}(x)\bigr)\bigr|\,d\sigma(\alpha)
\le
\|\psi\|_\infty.
$$
Hence, $P^*\psi$ is bounded.

To prove continuity, let $x_n \to x$ in $\mathbb{R}^d$. For each fixed $\alpha \in G_{k,d}$, since $\pi_{\alpha^\perp}$ and $\psi$ are continuous, we have
$$
\psi\bigl(\alpha,\pi_{\alpha^\perp}(x_n)\bigr)
\longrightarrow
\psi\bigl(\alpha,\pi_{\alpha^\perp}(x)\bigr)
$$
as $n \to \infty$. Moreover,
$$
\bigl|\psi\bigl(\alpha,\pi_{\alpha^\perp}(x_n)\bigr)\bigr|
\le \|\psi\|_\infty.
$$
Therefore, by the dominated convergence theorem (see e.g., \cite{Folland}),
$$
(P^*\psi)(x_n)
=
\int_{G_{k,d}} \psi\bigl(\alpha,\pi_{\alpha^\perp}(x_n)\bigr)\,d\sigma(\alpha)
\longrightarrow
\int_{G_{k,d}} \psi\bigl(\alpha,\pi_{\alpha^\perp}(x)\bigr)\,d\sigma(\alpha)
=
(P^*\psi)(x).
$$
Thus $P^*\psi$ is continuous on $\mathbb{R}^d$.

It remains to prove that $P^*\psi(x) \to 0$ as $|x|\to\infty$. Since $\psi\in C_0(\cG_{k,d})$ and $G_{k,d}$ is compact, the vanishing-at-infinity condition is uniform in $\alpha$: for every $\varepsilon>0$ there exists $R>0$ such that
\begin{align}
|\psi(\alpha,y)|<\varepsilon
\qquad\text{whenever } |y|>R,\ \alpha\in G_{k,d}.
\label{eq:uniform-psi}
\end{align}
For each $x\in\mathbb{R}^d$, define
$$
A(x):=\{\alpha\in G_{k,d}:|\pi_{\alpha^\perp}(x)|\le R\}.
$$
Then
\begin{align}
P^*\psi(x)
=
\int_{A(x)} \psi(\alpha,\pi_{\alpha^\perp}(x))\,d\sigma(\alpha)
+
\int_{A(x)^c}\psi(\alpha,\pi_{\alpha^\perp}(x))\,d\sigma(\alpha).
\label{eq:decomposition}
\end{align}
On $A(x)^c$, the estimate \eqref{eq:uniform-psi} gives
\begin{align}
\left|
\int_{A(x)^c}\psi(\alpha,\pi_{\alpha^\perp}(x))\,d\sigma(\alpha)
\right|
\le \varepsilon.
\label{eq:I2-bound-new}
\end{align}

We next show that $\sigma(A(x))\to 0$ as $|x|\to\infty$. Write $x=ru$, where $r=|x|$ and $u\in \mathbb S^{d-1}$. Then
$$
|\pi_{\alpha^\perp}(x)|
=
r|\pi_{\alpha^\perp}(u)|,
$$
and hence
\begin{align}
A(x)
=
\left\{\alpha\in G_{k,d}:|\pi_{\alpha^\perp}(u)|\le \frac{R}{r}\right\}.
\label{eq:A-x-explicit}
\end{align}
For fixed $u\in\mathbb S^{d-1}$, the map
$$
\alpha\mapsto |\pi_{\alpha^\perp}(u)|
$$
is continuous on the compact space $G_{k,d}$. The sets
$$
E_t(u):=\{\alpha\in G_{k,d}:|\pi_{\alpha^\perp}(u)|\le t\}
$$
are compact and decrease to
$$
\{\alpha\in G_{k,d}:u\in\alpha\},
$$
as $t$ goes to zero.

Since $1\le k\le d-1$, this set can be identified with $G_{k-1,d-1}$, hence is a proper submanifold of $G_{k,d}$ and has $\sigma$-measure zero. Therefore, by continuity of measure from above,
$$
\sigma(E_t(u))\to 0
\qquad\text{as } t\downarrow 0.
$$
Moreover, by the $O(d)$-invariance of $\sigma$, the quantity $\sigma(E_t(u))$ is independent of $u\in\mathbb S^{d-1}$. Hence the convergence is uniform in $u$. Applying this with $t=R/r$ in \eqref{eq:A-x-explicit}, we obtain
\begin{align}
\sigma(A(x))\longrightarrow 0
\qquad\text{as } |x|\to\infty.
\label{eq:A-measure-goes-to-0}
\end{align}

Using \eqref{eq:decomposition}, the bound $|\psi|\le \|\psi\|_\infty$, and \eqref{eq:A-measure-goes-to-0}, we get
\begin{align*}
\left|
\int_{A(x)} \psi(\alpha,\pi_{\alpha^\perp}(x))\,d\sigma(\alpha)
\right|
\le
\|\psi\|_\infty\,\sigma(A(x))
\longrightarrow 0.
\end{align*}
Together with \eqref{eq:I2-bound-new}, this gives
$$
\limsup_{|x|\to\infty}|P^*\psi(x)|\le \varepsilon.
$$
Since $\varepsilon>0$ is arbitrary, we conclude that
$$
P^*\psi(x)\to 0
\qquad\text{as } |x|\to\infty.
$$

Since $P^*\psi$ is continuous and vanishes at infinity, we conclude that
$P^*\psi\in C_0(\mathbb{R}^d)$.
\end{proof}

\section*{Acknowledgements}
The work of the first author was supported in part by NSF DMS grant 2206279. The work of the second author was supported in part by NSF DMS grant 230797. An AI-based language tool was used to assist in editing the manuscript for spelling, grammar, and stylistic improvements.

 \bibliographystyle{siam}
 \bibliography{references}

\end{document}